\documentclass[reqno]{amsart}

\usepackage{amsmath,amssymb,graphicx
}

\usepackage[latin1]{inputenc}
\newtheorem{Theorem}{Theorem}[section]
\newtheorem{proposition}[Theorem]{Proposition}
\newtheorem{definition}[Theorem]{Definition}
\newtheorem{corollary}[Theorem]{Corollary}
\newtheorem{lemma}[Theorem]{Lemma}
\theoremstyle{remark}
\newtheorem{Remark}[Theorem]{Remark}

\allowdisplaybreaks

\begin{document}

\title[Characterization of Cyclically Fully commutative elements]
{Characterization of Cyclically Fully commutative elements  in finite and affine Coxeter Groups}

\author[Mathias P\'etr\'eolle]{Mathias P\'etr\'eolle}
\address{Institut Camille Jordan, Universit\'e Claude Bernard Lyon 1,
69622 Villeurbanne Cedex, France}
\email{petreolle@math.univ-lyon1.fr}
\urladdr{http://math.univ-lyon1.fr/{\textasciitilde}petreolle}

\thanks{}

\date{}

\subjclass[2013]{}


\begin{abstract}
An element of a Coxeter group $W$ is fully commutative if any two of its reduced decompositions are related by a series of transpositions of adjacent commuting generators. An element of a Coxeter group $W$ is cyclically fully commutative if any of its cyclic shifts remains fully commutative. These elements were studied by Boothby \emph{et al}. In particular the authors enumerated cyclically fully commutative elements in all Coxeter groups having a finite number of them. In this work we characterize and enumerate cyclically fully commutative elements according to their Coxeter length in all finite or affine Coxeter groups by using a new operation on heaps, the cylindric transformation. In finite types, this refines the work of Boothby \emph{et al.}, by adding a new parameter. In affine type, all the results are new. In particular, we prove that there is a finite number of cyclically fully commutative logarithmic elements in all affine Coxeter groups. We study afterwards the cyclically fully commutative involutions and prove that their number is finite in all Coxeter groups.
\end{abstract}

\maketitle

\section*{Introduction}

Let $W$ be a Coxeter group. An element $w \in W$ is said to be fully commutative (FC) if any reduced word for $w$ can be obtained from any other one by only short braid relations. These elements were introduced and studied independently by Fan in \cite{FAN2}, Graham in \cite{GRA} and Stembridge in \cite{STEM1,STEM2,STEM3}, where among other things, he classified the Coxeter groups with a finite number of fully commutative elements and enumerated them in each case. Fully commutative elements appear naturally in the context of (generalized) Temperley-Lieb algebras, as they index a linear basis of the latter. Recently, in \cite{BJN2},  Biagioli, Jouhet and Nadeau refined Stembridge's enumeration by counting fully commutative elements according to their Coxeter length in any finite or affine Coxeter group.
\\ In this paper, we focus on a certain subset of fully commutative elements, the cyclically fully commutative  elements (which we denote CFC from now on).  These are elements for which every cyclic shift of any reduced expression is a reduced expression of a FC element. They were introduced by Boothby et al. in \cite{BBEEGM}, where they classified the Coxeter groups with a finite number of CFC elements (they showed that the latter are the groups with a finite number of FC elements) and enumerated them. The main goal of \cite{BBEEGM} is to establish necessary and sufficient conditions for a CFC element $w \in W$ to be logarithmic, i.e to satisfy $\ell(w^k)=k \ell(w)$ for any positive integer $k$. This is the first step towards studying a cyclic version of Matsumoto's theorem for CFC elements. Here we will focus on the combinatorics of CFC elements. We introduce a new operation on heaps, which will allow us to give a new characterization of CFC elements in all Coxeter groups (see Theorem~\ref{propcfc}). For finite or affine Coxeter groups, this characterization can be reformulated in terms of words, by using the work on FC elements from \cite{BJN2}. It allows us to enumerate the CFC elements  by taking into account their Coxeter length. We will also prove that the number of CFC involutions is finite in all Coxeter groups.

This paper is organized as follows.
 We recall in Section~1 some definitions and properties on Coxeter groups, we introduce a new operation on heaps, the \emph{cylindric transformation}, in two different ways, and we deduce a new characterization of CFC elements in terms of pattern-avoidance for cylindric transformed heaps (see Theorem~\ref{propcfc}).
In Section~2, we use this characterization in order to obtain a complete classification (in terms of words) of CFC elements in the affine group $\tilde{A}_{n-1}$. We also deduce a classification of CFC elements in the group $A_{n-1}$, and use this to enumerate CFC elements in both groups, according to their Coxeter length. The same work is done for the groups $\tilde{C}_n,~B_{n},~D_{n+1},~\tilde{B}_{n+1}$, and  $\tilde{D}_{n+2}$ in Section~3. In Section~4, we will focus on CFC involutions. The main result is that there is a finite number of CFC involutions in all Coxeter groups. We also give a characterization of CFC involutions for all Coxeter groups and enumerate them in finite and affine types.

\label{sec:intro}

\section{Cyclically fully commutative elements and heaps}
\label{sec:preliminaries}

\subsection{Cyclically fully commutative elements}
~\\Let $W$ be a Coxeter group with a finite generating set $S$ and Coxeter matrix $M=(m_{st})_{s,t \in S}$. Recall (see \cite{BB}) that this notation means that the relations between generators are of the form $(st)^{m_{st}}=1$ for $m_{s,t} \neq \infty$, $M$ is a symmetric matrix with $m_{ss}= 1$ and  $m_{st} \in \{ 2, 3, \ldots\} \cup \{  \infty \}$. The pair ($W, S$) is called a Coxeter system. We can write the relations as $sts\ldots =tst \ldots$, each side having length $m_{st}$; these are usually called braid relations. When $m_{st}=2$, this is a commutation relation or a short braid relation. We define the Coxeter diagram $\Gamma$ associated to ($W,S$) as the graph with vertex set $S$, and for each pair ($s,t$) with $m_{st} \geq 3$, an edge between $s$ and $t$ labelled by $m_{st}$. If $m_{st}=3$, we usually omit the labelling.
\\For $w \in W$, we define the Coxeter length $\ell(w)$ as the minimum length of any expression ${\bf w}=s_1 \ldots s_n$ with $s_i  \in S$. An expression is called reduced if it has a minimal length. The set of all reduced expressions of $w$ is denoted by $R(w)$. A classical result in Coxeter group theory, known as Matsumoto's theorem, is that any expression in $R(w)$ can be obtained from any other one using only braid relations. An element $w$ is said to be fully commutative  if any expression in $R(w)$ can be obtained from any other one using only commutation relations. For clarity, we will write $w$ for an element in $W$, and $\bf w$ for one of its words (or expressions).
\begin{definition} \label{defcfc}(\cite[Definition 3.4]{BBEEGM}) Given an expression ${\bf w}=s_{i_{1}} \dots s_{i_n}$, we define the left (\emph{resp.} right) cyclic shift of $\bf w$ as $s_{i_2} \dots s_{i_n} s_{i_1} $ (\emph{resp.} $s_{i_n} s_{i_1} \dots s_{i_{n-1}} $). A cyclic shift of $\bf w$ is  $s_{i_k} \dots s_{i_{k-2}} s_{i_{k-1}} $ for $k \in \mathbb{N}$. An element $w \in W$ is said to be cyclically fully commutative if every cyclic shift of any expression in $R(w)$ is a reduced expression for a fully commutative element.\end{definition}
We denote the set of CFC elements of $W$ by $W^{CFC}$.\\~
\subsection{Heaps and FC elements}
~\\We follow Stembridge~\cite{STEM1} in this section. Fix a word $\mathbf{w}=(s_{a_1},\ldots, s_{a_\ell})$ in $S^*$, the free monoid generated by $S$. We define a partial ordering of the indices $\{1,\ldots, \ell\}$ by $i\prec j$ if $i<j$ and $m(s_{a_i},s_{a_j})\geq 3$ and extend by transitivity on a partial ordering. We denote by $H_{\mathbf{w}}$ this poset together with the ``labeling'' $i\mapsto s_{a_i}$: this is the  \emph{heap} of $\mathbf{w}$. We will often call the elements of $H_{\bf w}$ points. In the Hasse diagram of $H_\mathbf{w}$, elements with the same labels will be drawn in the same column. The size $|H|$ of a heap $H$ is its cardinality. Given any
subset $I \subset S$, we will note $H_I$ the subposet induced by all elements of $H$ with labels in I. In particular $H_{s}:=H_{\{s\}}$ for $s \in S$ is the chain $H_s = s^{(1)} \prec s^{(2)} \prec \cdots \prec s^{(k)}$ where $k = |H_s |$. We also denote by  $|\mathbf{w}_s|$ the number $|H_s|$ (note that this also counts the number of occurences of s in {\bf w}).  In $H$, we say that there is a {\emph{chain covering relation}} between $i$ and $j$, denoted by $i \prec_c j$, if $ i \prec j $ and one of the two following conditions is satisfied:
\begin{itemize}
\item $m_{s_{a_i}s_{a_j}}\geq 3$ and there is no element $z$ with the same label as $i$ or $j$ such that $i \prec z \prec j$,
\item  $s_{a_i}= s_{a_j}$ and there is no element $z$ such that $i \prec z \prec j$.
\end{itemize}
Note that the set of chain covering relations corresponds to the set of edges in the corresponding heap. In other words, this set is equal to the set of covering relations of all subposets $H_{\{s,t\}}$, with $m_{st} \geq 3$.  In Figure~\ref{exheap}, we fix a Coxeter diagram on the left, and we give two examples of words with the corresponding heaps.

\begin{figure}[!h]

\includegraphics[scale=0.99]{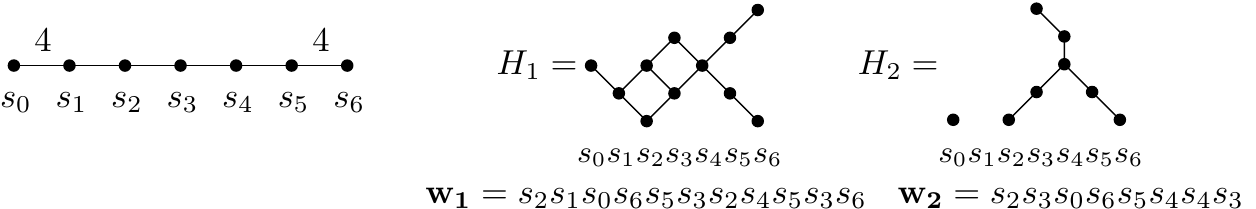}
\caption{\label{exheap}Two different heaps.}

\end{figure}

If we consider heaps up to poset isomorphisms which preserve the labels, then heaps encode precisely {\em commutativity classes}, that is, if the word ${\bf w}'$ is obtainable from $\bf w$ by transposing commuting generators then there exists a poset isomorphism between $H_{\mathbf{w}}$ and $H_{\mathbf{w'}}$. In particular, if $w$ is fully commutative, the heaps of the reduced words are all isomorphic, and thus we can define the heap of $w$, denoted $H_w$.

Let ${\bf w}=(s_{a_1},\ldots, s_{a_\ell})$ be a word. A \emph{linear extension} of the poset $H_{\bf w }$ is a linear ordering $\pi$ of $\{1, \ldots, \ell \}$ such that $\pi(i)<\pi(j)$ implies $i\prec j$. Now let $\mathcal{L}(H_{\bf w})$ be the set of words $(s_{a_{\pi(1)}},\ldots, s_{a_{\pi(\ell)}})$ where $\pi$ ranges over all linear extensions of $H_{\mathbf{w}}$.
\begin{proposition}[\cite{STEM1}, Proposition 1.2] Let $w$ be a fully commutative element. Then
$\mathcal{L}(H_{\mathbf{w}})$ is equal to $\mathcal{R}(w)$ for some (equivalently every)  ${\mathbf{w}} \in \mathcal{R}(w)$.
\end{proposition}
We say that a chain $i_1\prec \cdots \prec i_m$ in a poset $H$ is convex if the only elements $u$ satisfying $i_1\preceq u\preceq i_m$ are the elements $i_j$ of the chain. The next result characterizes {\em FC heaps}, namely the heaps representing the commutativity classes of FC elements.

\begin{proposition}[\cite{STEM1}, Proposition 3.3]
\label{propfc}
A heap $H$ is the heap of some FC element if and only if $(a)$ there is no convex chain $i_1\prec_c \cdots\prec_c i_{m_{st}}$ in $H$ such that $s_{i_1}=s_{i_3}=\cdots=s$ and   $s_{i_2}=s_{i_4}=\cdots=t$ where $3\leq m_{st}<\infty$, and $(b)$ there is no chain covering relation $i\prec_c j$ in $H$ such that $s_i=s_j$.
\end{proposition}

\subsection{Cylindric transformation of heaps and CFC elements}
Now, we will focus on CFC elements.  Before this, we need to define a new operation on heaps.

\begin{definition}\label{def} Let $H:=H_{\bf w}$ be a heap of a word ${\bf w}=s_{a_1} \ldots s_{a_\ell}$. The cylindric transformation $ H^c$ of $H$ is the labelling $i\mapsto s_{a_i}$ and a relation on the indices $\{1, \ldots , \ell\}$, made of the chain covering relations $\prec_c$ of $H$, together with some new relations defined as follows:
\begin{itemize}
\item for each generator $s$, consider the minimal point a and the maximal point b in the chain $H_{s}$ (for the partial order $\prec$). If a is minimal and b is maximal in the poset $H$, we add a new relation $b \prec_c a$.
\item for each pair of generators (s,t) such that $m_{s,t}\geq 3$, consider the minimal point a and the maximal point b in the chain $H_{\{s,t\}}$ (for the partial order $\prec$). If these points have different labellings (one has label s and the other has label t), we add a new relation $b \prec_c a$.
\end{itemize}
\end{definition}

 \noindent \emph{Example:} for the simply laced linear Coxeter diagram with 7 generators, consider the word ${\bf w} = s_2s_1s_0s_3s_2s_6s_5s_4s_5s_6s_3$. Its heap and its cylindric transformation are represented in Figure~\ref{excylindric}.
\begin{figure}[!h]

\includegraphics[scale=1]{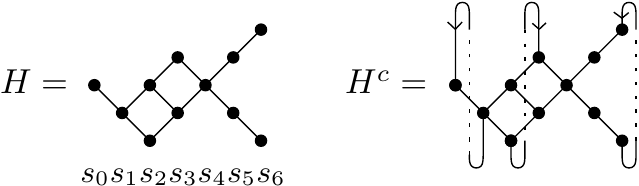}
\caption{\label{excylindric}A heap and its cylindric transformation.}

\end{figure}

\begin{Remark}
$H^c$ is not a poset, even if we extend the relation $ \prec_c$ by transitivity. Indeed, each point in $H^c$ would otherwise be in relation with itself. As $H^c$ is not a poset, we cannot draw its Hasse diagram.  We define the diagram of $H^c$ as the Hasse diagram of $H$, in which we add oriented edges representing the new relations described in Definition~\ref{def}.
\end{Remark}

Let us explain the name ``cylindric transformation" that we decided to use. Consider the Coxeter system ($W, S$) corresponding to the linear Coxeter
diagram $\Gamma_n$ of Figure \ref{Linear}.

\begin{figure}[!h]
\includegraphics{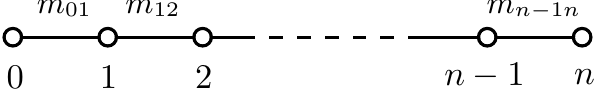}
\caption{\label{Linear}The linear Coxeter diagram.}
\end{figure}

 In this case, when we draw the diagram of $H^c$, it should be seen as drawn on a cylinder instead of a plane as for Hasse diagrams of heaps (in our cylinder, each chain $H_s$ for a generator s can be seen as a circle).

As for heaps, if $H^c$ is a cylindric transformation of a heap, the size $|H^c|$  is its cardinality. Given any subset $I \subset S$, we will write $H_I^c$ for the set of points in  $H^c$ with labels in $I$. Unlike in the definition of $H_I$, this definition does not include the relations $\prec_c$ between those points.

We will now give the construction of the cylindric transformation of a heap in terms of words.

\begin{proposition}\label{propositioncfc}
If ${\bf w}=s_{a_1}\ldots s_{a_k}$ is a word and H its heap, let  ${\bf w}_1=s_{a_1}\ldots s_{a_p}$ be the shortest prefix of $\bf w$ such that all generators which appear in $\bf w$ also appear in ${\bf w}_1$. We write ${\bf w' = w w_1}:=s_{c_1}\ldots s_{c_{k+p}}$ where we denote by $ c_i:= a_i,~1 \leq i \leq k$ and $c_{k+i}:=a_i,~1 \leq i \leq p$. Consider the heap $H_{{\bf w}'}$ of the word {\bf w'}. In $H_{{\bf w}'}$ we keep only the points and the chain covering relations, we identify the points $i$ and $k+i$ for $1 \leq i \leq p$, and we denote the resulting diagram by $H'$. Then we have $H' =H^c$.
\end{proposition}

\begin{Remark}Note that $H'$ define in Proposition~\ref{propositioncfc} is not a poset, although $H_{{\bf w'}}$ is.
\end{Remark}

\begin{figure}[!h]
\includegraphics[scale=1.2]{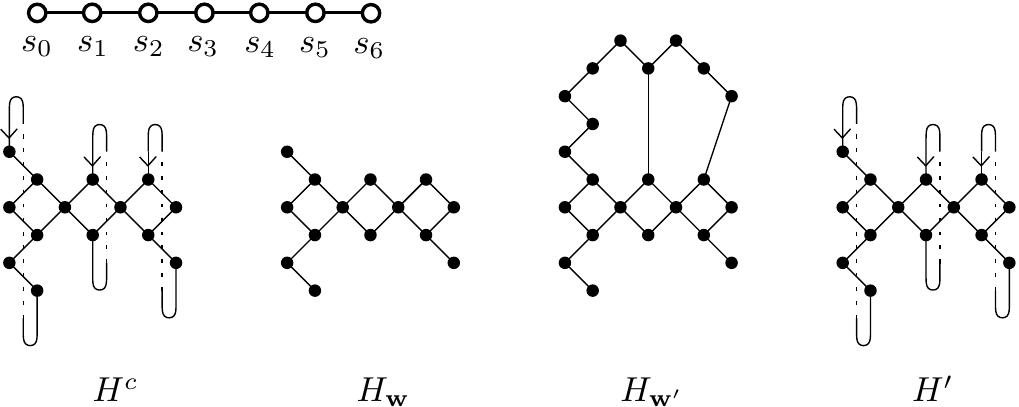}
\caption{ Illustration of Proposition~\ref{propositioncfc} for 
${\bf w}= s_1s_0s_1s_3s_2s_6s_5s_4s_6s_5s_3s_0s_1s_0$.}
\hspace*{-3.4cm} We have ${\bf w}_1= s_1s_0s_1s_3s_2s_6s_5s_4$.
\end{figure}

\begin{proof} We will first prove that the chain covering relations $\prec_c$ in $H'$ are included in those of $H^c$.  There are three types of chain covering relations in $H_{\bf w'}$:
\begin{itemize}
\item The ones between two points $i$ and $j$ such that $1\leq i,j\leq k$. They correspond to chain covering relations which also appear in $H$, and are therefore present in $H^c$.
\item The ones between two points $i$ and $j$ such that $k+1\leq i,j\leq k+p$. After the identification between $i$ (\emph{resp.} $j$) and $i-k$ (\emph{resp.} $j-k$), these relations also appear in $H$, and so in $H^c$.
\item The ones between two points $i$ and $j$ such that $1 \leq i \leq k$ and $k+1 \leq j \leq p$. By construction of $H'$, the point $j$ will be identified with the point $h := j-k$ which is minimal in $H_{s_{a_j}}$ (and even minimal in $H$ if $s_{a_j}=s_{a_i}$). Likewise, the point i is maximal in $H_{s_{a_i}}$ (and even maximal in $H$ if $s_{a_j}=s_{a_i}$). So, when we construct $H^c$, we will add a chain covering relation  $i \prec_c h$.
\end{itemize}
Thus the chain covering relations of $H'$ are included in those of $H^c$. The reverse inclusion can be proved by similar arguments.
\end{proof}

\begin{proposition}\label{injectif} The canonical mapping $ H \mapsto H^c$ is not injective. More precisely, if the word $\bf w_2$ is commutation equivalent to some cylic shift of the word $\bf w_1$, then $H_{\bf w_1}^c= H_{\bf w_2}^c$.
\end{proposition}

\begin{proof}
For $s$ a generator, $\bf{w}$  a word, define $H:=H_{s{\bf w}}$ and $H':=H_{{\bf w}s}$. We will prove that the chain covering relations of $H^c$ are included in those of $H'^c$. By applying this to the successive right cyclic shifts of ${\bf w}s$, we may conclude that $H^c= H'^c$, which is sufficient to prove the proposition.
\\We denote by $i$ (\emph{resp.} $i'$) the point of $H$ (\emph{resp.} $H'$) corresponding to the first (\emph{resp.} last) occurrence of s in $s{\bf w}$ (\emph{resp.} ${\bf w}s$). There are three types of chain covering relations in $H^c$:
\begin{itemize}
\item The ones between two points different from $i$; they are also present in $H'^c$.
\item The ones between $i$ and another point $j$. Two cases can occur: $i \prec_c j$ in $H $ or not. In the first case, $j$ must be minimal among all points with the same label. As $i'$ is maximal in $H'$, we add a relation $i' \prec_c j$ in $H'^c$ when we construct $H'^c$. The second case occurs when we add a relation between $i$ and $j$ during the construction of $H^c$. In this case, $j$ is maximal in $H$ among all points with the same label and so there is a relation $j \prec _c i'$ in $H'$. This relation stays naturally present in $H'^c$.
\item The ones between $i$ and itself. It is clear that $i'$ is also in relation with itself in $H'^c$.
\end{itemize}
\end{proof}

Now we characterize the CFC elements in terms of heaps.

\begin{definition}Set $H$ a heap and $s$, $t$ two generators such that $m_{st} \geq 3$. Consider a chain $i_1 \prec_c \cdots  \prec_c i_m$ in $H^c$ with $m \geq 3$, all elements $i_k$ distinct and such that $s_{a_{i_1}}=s_{a_{i_3}}=\cdots=s$ and $s_{a_{i_2}}=s_{a_{i_4}}=\cdots=t$.\\ Such a chain is called cylindric convex if the only elements $u_1, \ldots, u_d$, satisfying $i_1 \prec_c \cdots \prec_c i_k \prec_c u_1 \prec_c \cdots \prec_c u_d \prec_c i_m$ with all elements involved in this second chain distinct, are the elements $i_j$ of the first chain. 
\end{definition}
\begin{lemma}\label{Remark1} Assume there exists a convex chain $i_1 \prec_c \cdots  \prec_c i_m$ of length $m\geq 3$ in a heap $H$, such that all elements $i_k$ are distinct and $s_{a_{i_1}}=s_{a_{i_3}}=\cdots=s$ and $s_{a_{i_2}}=s_{a_{i_4}}=\cdots=t$. Then this chain is also cylindric convex in the cylindric transformation $H^c$.
\end{lemma}
\begin{proof}
Let $i_1 \prec_c \cdots  \prec_c i_m$ be such a convex chain in $H$ and assume there is a chain $i_1 \prec_c \cdots \prec_c i_k \prec_c u_1 \prec_c \cdots \prec_c u_d \prec_c i_m$ with all elements distinct in $H^c$, and $u_1 \neq i_{k+1}$. Two cases can occur:
\begin{itemize}
\item $i_k \neq i_{m-1}$: as $i_k$ is not a maximal element in $H_{s_{a_{i_k}}}$, the relation $i_k \prec_c u_1$ holds in $H$. Consequently, the relation $ i_k \prec_c u_1 \prec i_{k+2}$ holds in $H$. So $i_1 \prec_c \cdots  \prec_c i_m$ is not a convex chain in $H$, which is a contradiction,
\item $i_k = i_{m-1}$: as $i_m$ is not a minimal element in $H_{s_{a_{i_m}}}$, the relation $u_d \prec _c i_m$ holds in $H$. Consequently, the relation $ i_{m-2} \prec u_d \prec_c i_m$ holds in $H$. So, as $u_d \neq i_{m-1}$, $i_1 \prec_c \cdots  \prec_c i_m$ is not a convex chain in $H$, which is a contradiction.
\end{itemize}
\end{proof}

\begin{Theorem}\label{propcfc}
A heap $H$ is the heap of some word $\bf w$ corresponding to a CFC element w if and only if both the following conditions are satisfied:
\begin{itemize}
\item there is no cylindric convex chain $i_1\prec_c \cdots \prec_c i_{m_{st}}$ in the cylindric transformation $H^c$ such that $s_{a_{i_1}}=s_{a_{i_3}}=\cdots=s$ and   $s_{a_{i_2}}=s_{a_{i_4}}=\cdots=t$, where $3\leq m_{st}<\infty$;
\item there is no chain covering relation $i\prec_c j$ in the cylindric transformation $H^c$ such that $s_{a_i}=s_{a_j}$.
\end{itemize}
\end{Theorem}

\begin{proof} Let $\bf w$ be a word and $H:= H_{\bf w}$ its heap. Assume that $\bf w$ is a reduced word of a non CFC element. There exists by Definition~\ref{defcfc} a cyclic shift of $\bf w$ which is commutation equivalent to ${\bf w'}= w_1 ssw_2$ or ${\bf w'}= w_1\underbrace{st \cdots  st}_{ m_{st}} w_2$. Let $H'$ be the heap of $\bf w'$. By Proposition~\ref{propfc}, in the first case, $H'$ contains a chain covering relation $i\prec_c j$ such that $s_{a_i}=s_{a_j}=s$, and in the second case, $H'$ contains a convex chain $i_1\prec_c \cdots\prec_c i_{m_{st}}$   such that $s_{a_{i_1}}=s_{a_{i_3}}=\cdots=s$ and   $s_{a_{i_2}}=s_{a_{i_4}}=\cdots=t$, where $3\leq m_{st}<\infty$. By Lemma~\ref{Remark1}, these two cases give a chain covering relation and a cylindric convex chain in $H'^c$, respectively. But ${\bf w'}$ is commutation equivalent to a cyclic shift of $\bf w$, so $H^c=H'^c$, as seen in Proposition~\ref{injectif}.
\\Conversely, if $H^c$ contains the cylindric convex chain $i_1 \prec_c \cdots\prec_c  i_{m_{st}}$, let $\bf w'$ be the cyclic shift of $\bf w$ beginning by $s_{i_1}$, and let $H'$ be its heap. As all elements in the cylindric convex chain are distinct, $H'$ contains the convex chain $i_1\prec_c \cdots\prec_c  i_{m_{st}}$. Therefore $H'$ is not FC by Proposition~\ref{propfc} and $H$ is not CFC. The same argument also holds if $H^c$ contains a relation $i\prec_c j$ such that $s_{a_i}=s_{a_j}$, by letting this time $\bf w'$ be the cyclic shift of $\bf w$ beginning by $s_{a_i}$.\end{proof}

For example, this theorem ensures that the heap in Figure~\ref{excylindric} does not correspond to a CFC element, as it contains a chain covering relation $i\prec_c j$ in the cylindric transformation $H^c$ such that $s_{a_i}=s_{a_j}= s_6$. The example in Figure \ref{excfc} corresponds to a CFC element, according to Theorem~\ref{propcfc}.

\begin{figure}[!h]
\includegraphics[scale=1]{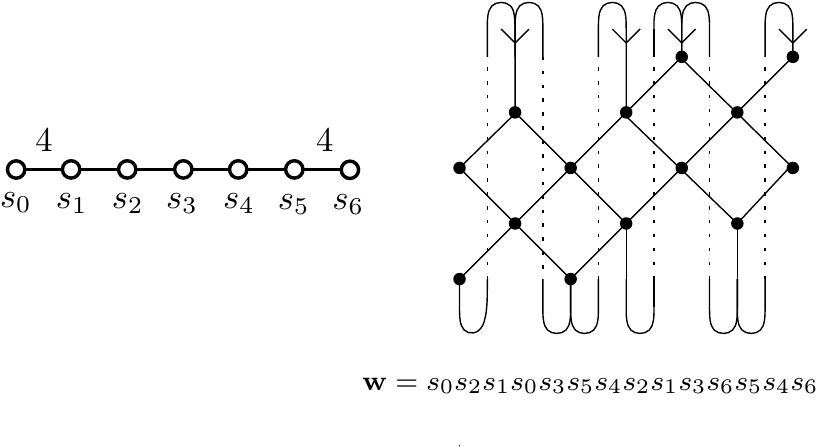}
\caption{\label{excfc}A CFC element in $\tilde{C}_6$.}
\end{figure}

Before ending this section, let us recall the set of alternating heaps which was defined in \cite{BJN2}, and will be useful later.

\begin{definition}\label{alternating}
In a Coxeter group with linear diagram $\Gamma_n$, a reduced word is alternating if for $i=0,1,\ldots, n-1$, the occurrences of $s_i$ alternate with those of $s_{i+1}$. A heap is called alternating if it is the heap of an alternating word. If the Coxeter group is $\tilde{A}_{n-1}$ (see Figure \ref{dynkina} for the Coxeter diagram), the diagram is not linear but  we define the alternating word in the same way by setting $s_0= s_n$.
\end{definition}

\begin{Remark}\label{generatoronce} The subset of alternating words defined as the set of words in which each generator occurs at most once is particular, in the sense that each element of this subset is a reduced expression for a CFC element (we can not use any braid relation of length at least 3). This fact justifies why this subset will be treated separately from other alternating words in the rest of the article.

\end{Remark}

\section{CFC elements in types $\tilde{A}$ and $A$}\label{finite}

In this section, we will give a characterization and the enumeration of CFC elements in both types $A_{n-1}$ and $\tilde A_{n-1}$. The corresponding Coxeter diagrams are given in Figure~\ref{dynkina}.
\begin{figure}[!h]
\begin{center}\includegraphics{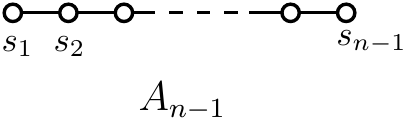}\hspace*{1cm}\includegraphics{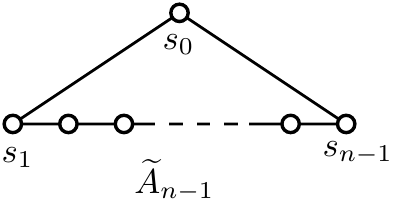}
\end{center}
\caption{Coxeter diagrams of types $A_{n-1}$ and $\tilde{A}_{n-1}$.}
\label{dynkina}
\end{figure}
\\The elements of $A_{n-1}^{CFC}$ were enumerated in \cite{BBEEGM} by using recurrence relations. Our characterization in terms of heaps allows us to take into account the Coxeter length. Actually, we can compute the generating functions $W^{CFC}(q) = \sum_{w \in W^{CFC}}q^{\ell(w)}$ for $W=A_{n-1}$ and $\tilde{A}_{n-1}$. In particular, when $q=1$ and $W=A_{n-1}$, we get back the enumeration from \cite{BBEEGM} (recall that the number of CFC elements in type $\tilde{A}_{n-1}$ is infinite). Our strategy is the following: first, we obtain a characterization of CFC elements in type $\tilde{A}_{n-1}$, we deduce from it a characterization of CFC elements in type $A_{n-1}$, then we derive  the enumeration of CFC elements in type $A_{n-1}$ and deduce from it the enumeration in type $\tilde{A}_{n-1}$.

\subsection{Characterization in type $\tilde{A}_ {n-1}$}

Note that, in this type,  the diagram of a cylindric transformation of a heap can be seen as drawn on a torus.
\begin{Theorem}\label{thmantilde} An element $ w\in \tilde A_{n-1}$ is CFC  if and only if one (equivalently, any) of its reduced expressions $\bf w$ verifies one of these conditions:
\begin{itemize}
\item[(a)] each generator occurs at most once in $\bf w$,
\item[(b)] $\bf w$ is an alternating word and  $|{\bf w}_{s_0}|=|{\bf w}_{s_1}|= \cdots=|{\bf w}_{s_{n-1}}| \geq 2$.
\end{itemize}
\end{Theorem}

\begin{proof} As said in Remark \ref{generatoronce}, if each generator occurs at most once in $\bf w$, then $w$ is a CFC element. So let $\bf w$ be a reduced expression of a CFC element $w$ having a generator  occuring at least twice in $\bf w$. Recall that according to \cite{BJN2}, $w \in \tilde A_{n-1}$ is fully commutative if and only if $\bf w$ is an alternating word. Let $s_j$ be a generator which occurs at least twice in $\bf w$ and such that for all $ k \in \{0,1, \ldots , n-1 \}, |{\bf w}_{s_j}| \geq |{\bf w}_{s_k}|$. We will prove that $|{\bf w}_{s_j}|= |{\bf w}_{s_{j+1}}|$ where we set $s_n=s_0$, which is sufficient to show that each generator occurs the same number of times. As $\bf w$ is alternating,  there are three possibilities:
\begin{itemize}
\item $|{\bf w}_{s_j}|= |{\bf w}_{s_{j+1}}| -1$. It contradicts the maximality of $|{\bf w}_{s_j}|$.
\item $|{\bf w}_{s_j}|= |{\bf w}_{s_{j+1}}| +1$. By maximality of $|{\bf w}_{s_j}|$, there are two possibilities for  $H_{{\bf w}}^c$: $|{\bf w}_{s_j}|= |{\bf w}_{s_{j-1}}|$ or $|{\bf w}_{s_j}|= |{\bf w}_{s_{j-1}}| +1$. We obtain either a cylindric convex chain $x \prec_c y \prec_c z$ where $x$ and $z$ have label $s_j$ and $y$ has label $s_{j-1}$, or a chain covering relation between two indices  $p$ and $q$ with label $s_j$, which  allows us to conclude that $w$ is not a CFC element by Theorem \ref{propcfc} (see the figure below, where we have circled the points $x$, $y$, $z$, and $p$, $q$ respectively).
\newpage
\begin{figure}[!h]
\includegraphics{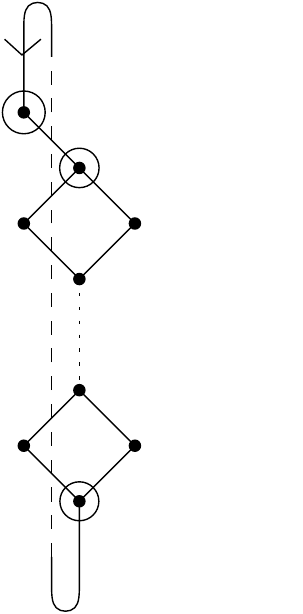}\hspace*{2cm}\includegraphics{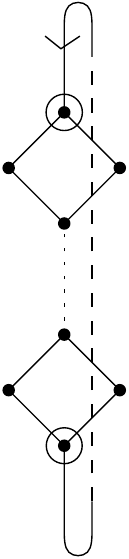}
\caption{The two possible heaps.}
\end{figure}

\item $|{\bf w}_{s_j}|= |{\bf w}_{s_{j+1}}|$ which is the needed condition.
\end{itemize}
Conversely, let   $\bf w$ be an alternating word such that  $|{\bf w}_{s_0}|=|{\bf w}_{s_1}|= \cdots=|{\bf w}_{s_{n-1}}| \geq 2$. $H_{\bf w}^c$ can not contain a cylindric convex chain $x  \prec_c y \prec_c z$ of length 3   such that $s_{a_x}=s_{a_z}= s_m $ and $s_{a_y}= s_{m+1}$ (\emph{resp.} $s_{m-1}$): indeed, the required condition on $\bf w$ implies that there exists an index $\ell $ such that $x  \prec_c \ell \prec_c z$ with $s_{a_\ell}= s_{m-1}$ (\emph{resp.} $s_{m+1}$), which is a contradiction with the cylindric convexity of the chain. The same argument holds for chain covering relations involving indices with the same labellings. \end{proof}

\subsection{Characterization and enumeration in type $A_{n-1}$}
In this case, we will both characterize the CFC elements and compute the generating function $$A^{CFC}(x):=\sum_{n \geq 1}A_{n-1}^{CFC}(q)x^{n}.$$

We begin with a lemma which is a consequence of Corollary 5.6 in \cite{BBEEGM}. Nevertheless, we will give a shorter proof using Theorem \ref{thmantilde}.
\begin{lemma}\label{lemmean}The CFC elements in type $A_{n-1 }$ are those having reduced expressions in which each generator occurs at most once.
\end{lemma}
\begin{proof} Let $\bf w$ be a reduced word of a CFC element in type $A_{n-1}$. By definition of the Coxeter diagram (cf Figure~\ref{dynkina}), it is a reduced word of a CFC element in type $\tilde{A}_{n-1}$ in which the generator $s_0$ does not occur. According to Theorem \ref{thmantilde}, the only such CFC elements in  type $\tilde{A}_{n-1}$ are those in which each generator occurs at most once. Conversely, if all generators occur at most once in $\bf w$,  Remark~\ref{generatoronce} ensures that $\bf w$ is a reduced expression of a CFC element.
\end{proof}

\begin{Theorem}\label{thman}
 We have $A_0^{CFC}(q)=1$, $A_1^{CFC}(q)=1+q$ and for $n \geq 3$,
 \begin{equation}\label{equationa2}A_{n-1}^{CFC}(q) = (2q+1)A_{n-2}^{CFC}(q)-qA_{n-3}^{CFC}(q).\end{equation}

Equivalently, we have the generating function:
 $$A^{CFC}(x)= x\frac{1-qx}{1-(2 q+1)x+qx^2}.$$
 \end{Theorem}
 
\begin{proof} According to \cite[Proposition 2.7]{BJN2},  FC elements in $A_{n-1}$ are in bijection with Motzkin type paths of length $n$, with starting and ending points at height 0, where the horizontal steps are labeled either L or R (and horizontal steps at height 0 are always labeled R). We recall the bijection, which is defined as follows: let $w$ be a FC element in $A_{n-1}$, and $H$ its heap. To each $s_i \in S$, we associate a point $P_i=(i, |H_{s_i}|)$. As $w$ is alternating, three cases can occur:
 \begin{itemize}
 \item $|H_{s_i}| =|H_{s_{i+1}}|-1$, corresponding to an ascending step,
 \item $|H_{s_i}|=|H_{s_{i+1}}|+1$, corresponding to a descending step,
 \item $|H_{s_i}| =|H_{s_{i+1}}|$, corresponding to an horizontal step, labelled by R if $s_i$ occurs before $s_{i+1}$ in $w$, L otherwise.
\end{itemize}   
 According to Lemma~\ref{lemmean}, the restriction of this bijection to CFC elements is a bijection between CFC elements and the previous Motzkin paths, having length $n$, whose height does not exceed 1. By taking into account the first return to the $x$-axis, we obtain the following recurrence relation for $n \geq 3$:\

\begin{equation}\label{equationa}A_{n-1}^{CFC}(q) = A_{n-2}^{CFC}(q)+ \sum_{m=2}^{n} 2^{m-2} q^{m-1} A_{n-1-m}^{CFC}(q),
\end{equation}
 where we write $A_{-1}^{CFC}(q)=1$ (which fits with $A_0^{CFC}(q)=1$, $A_1^{CFC}(q)=1+q$ and the expected recurrence relation). Rewriting (\ref{equationa}) with $n$ replaced by $n-1$ and combining with (\ref{equationa})  allows us to eliminate the sum over $m$, and leads to the recurrence relation (\ref{equationa2}).

Classical techniques in generating function theory and the values $A_0^{CFC}(q)=1$, $A_1^{CFC}(q)=1+q$
enable us to compute the desired generating function. \end{proof}

Notice that, as expected, if $q\rightarrow1$, we find the odd-index Fibonacci numbers generating function of \cite{BBEEGM}. This $q$-analog was already known, see sequence A105306 in \cite{SLO}.

\subsection{Enumeration in type $\tilde{A}_{n-1}$}

\begin{proposition} We have for $n \geq 3$
\begin{equation}\label{equationatilde}\tilde{A}_{n-1}^{CFC}(q) = P_{n-1}(q) + \frac{2^n-2}{1-q^n}~q^{2n},\end{equation}
where $P_{n-1}(q)$ is a polynomial of degree n satisfying for n$ \geq 4$
$$P_{n}(q)=(3q+1)P_{n-1}(q)+(2q+2q^2)P_{n-2}(q)-q^2P_{n-3}(q),$$
 with $P_{1}(q)=1+2q+2q^2$, $P_{2}(q)=1+3q+6q^2+6q^3$, and $P_3(q)=1+4q+10q^2+16q^3+14q^4$. Moreover, we can compute the generating function:
 $$P(x) :=\sum_{n=1}^\infty P_n(q) x^n = \frac{x(1+2q+2q^2-(2q+2q^2)x+q^2x^2)}{(1-qx)(1-(2q+1)x+qx^2)}.$$
 
Therefore the coefficients of $\tilde{A}_{n-1}^{CFC}(q)$ are ultimately periodic of exact period n, and the periodicity starts at length n.
\end{proposition}

\begin{proof}  In the same way as for the finite type $A_{n-1}$ (see \cite{BJN2}), FC elements  in $\tilde{A}_{n-1}$ are in bijection with Motzkin type paths of length n satisfying the following conditions:
\begin{itemize} 
 \item the starting point $P_0=(0, |H_{s_0}|)$ and the ending point $P_n=(n, |H_{s_n}|=|H_{s_0}|)$ have the same height,
 \item horizontal steps at height 0 are always labeled R,
 \item if the path contains only horizontal steps at height $\geq$ 1, the steps must  have not all the same labelling.
 \end{itemize} 
The construction of the path corresponding to a FC element is the same as in type $A$ if we set $s_n=s_0$.
In Theorem \ref{thmantilde}, the alternating CFC elements in which  all generators occur at least twice and in the same number  correspond to Motzkin type paths which have only horizontal steps. Therefore there are $2^n-2$ such paths for all fixed starting height $h \geq 2$.
It leads to the generating function
 $$\sum_{h=2}^{+\infty}(2^n-2)(q^{n})^h,$$
which can be summed to obtain the second term of the right hand side of (\ref{equationatilde}). In Theorem \ref{thmantilde}, the CFC elements which correspond to reduced expressions with at most one occurrence of each generator correspond to Motzkin type paths which stay at height $\leq 1$.  We denote by $P_{n-1}(q)$ the generating function of such elements. Let $i$ (\emph{resp.} $j$) be the first (\emph{resp.} last) return to the x-axis (see Figure~\ref{motzkinpathsatilde} for an example).

\begin{figure}[!h]
\begin{center}\includegraphics[scale=0.8]{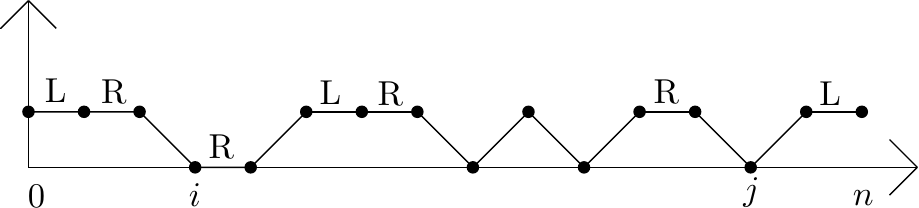}
\end{center}
\caption{Motzkin path corresponding to the element $s_{14}s_1s_0s_2s_6s_5s_7s_9s_{12}s_{11}$.}
\label{motzkinpathsatilde}
\end{figure}
We get
\begin{align}\label{equationatilde2}P_{n-1}(q) =& (2^n-2)q^n+nq^{n-1}2^{n-2} \nonumber \\&+\sum_{\genfrac{}{}{0pt}{}{i=0}{i \neq 0 \text{ or}}  }^{n-2}\sum_{\genfrac{}{}{0pt}{}{j=i+1}{j \neq n-1}}^{n-1} q^{n-1+i-j}2^{n-2+i-j} A_{j-i-1}^{CFC}(q) +A_{n-2}^{CFC}(q),
\end{align}
where the first term counts paths which stay at height 1, the second term counts paths such that $i=j$, and the last term counts paths such that $i=0$ and $j=n-1$. Rewriting (\ref{equationatilde2}) with $n-1$ replaced by $n-2$, combining with (\ref{equationatilde2}) and using the relation~\eqref{equationa} in Theorem \ref{thman} allow us to eliminate the double sum, and leads to the following 
relation:
$$P_{n-1}(q)=2qP_{n-2}(q)+2q^n +A_{n-1}^{CFC}(q)-qA_{n-3}^{CFC}(q).$$
From this and \eqref{equationa2}, it is easy to derive the expected relation and the generating function P. As $P_{n-1}(q)$ is a polynomial in q, the periodicity is clear, and its beginning comes from the fact that the leading coefficient of $P_{n-1}(q)$ is $(2^n -2)q^n$ and $[q^{n-1}]P_{n-1}(q) \neq 0$, where $[q^{n-1}]P_{n-1}(q)$ is the coefficient of $q^{n-1}$ in $P_{n-1}(q)$. \end{proof}

The situation in type $\tilde{A}_{n-1}$ is very different from all other types that we study, as we will see later: this is the only case where the generating series $ P_n$ of CFC elements which correspond to reduced expression with at most one occurrence of each generator  do not satisfy the recurrence relation $f_n(q) = (2q+1)f_{n-1}(q)-qf_{n-2}(q)$.

\section{Types $\tilde{C}, \tilde{B}, \tilde{D}, B$ and $D$.}\label{affine}
There are three other infinite families of affine Coxeter groups, they correspond to types $ \tilde B,\tilde C,\tilde D.$ All these groups contain an infinite number of CFC elements. In  any of these cases, using Theorem \ref{propcfc}, we are able to characterize the CFC elements and to compute the generating function $ W^{CFC}(q) = \sum_{w \in W^{CFC}}q^{\ell(w)}$. We also show that this generating function is always ultimately periodic. As a corollary, we deduce a characterization and the enumeration of CFC elements in finite types $B$ and $D$.

\subsection{Types $\tilde{C}_ {n}$ and $B_{n}$}

The Coxeter diagram of type $\tilde{C}_{n}$ is represented below.
\begin{figure}[!h]
\includegraphics{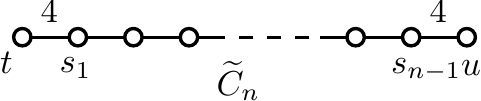}
\caption{Coxeter diagram of type $\tilde{C}_n$.}
\end{figure}

The situation is more complicated than in type $\tilde{A}_ {n-1}$, but we obtain an explicit characterization of CFC elements.

\begin{Theorem}\label{thmctilde}An element $w\in \tilde C_{n}$ is CFC if and only if one (equivalently any) of its reduced expressions $\bf w$ verifies one of the three conditions:
\begin{itemize}
\item[(a)] each generator occurs at most once in $\bf w$, 
\item[(b)] $\bf w$ is an alternating word and  $|{\bf w}_{t}|=|{\bf w}_{s_1}|= \cdots=|{\bf w}_{s_{n-1}}|=|{\bf w}_u| \geq 2$,
\item[(c)] $\bf w$ is a subword of the infinite periodic word $(ts_1s_2 \cdots s_{n-1} u s_{n-1} \cdots s_2 s_1)^ {\infty}$, where $|{\bf w}_{s_1}|= \cdots=|{\bf w}_{s_n-1}| \geq 2$ and $|{\bf w}_{t}|=|{\bf w}_{u}|=|{\bf w}_{s_1}|/2$ (i.e  we have $ {\bf w} =s_is_{i+1} \cdots s_{i-2}s_{i-1}$ or $s_is_{i-1} \ldots s_{i+2}s_{i+1}$, where $s_0=t$ and $s_n=u$).
\end{itemize}
\end{Theorem}

\begin{proof}
Let $w$ be a CFC element in $\tilde C_{n}$ and $\bf w$  be one of its reduced expressions. We denote by $H$ the heap of {\bf w} and $H^c$ its cylindric transformation.
In \cite[Theorem~3.4]{BJN2}, FC elements are classified in five families, the first corresponding to alternating elements. As before, we distinguish two cases for elements in this first family:
\begin{itemize}
\item each generator occurs at most once in $\bf w$. According to Remark~\ref{generatoronce}, $w$ is CFC as no long braid relations can be applied. These elements satisfy (a).
\item $\bf w$ is an alternating word in which a generator occurs at least twice. In this case, the proof will essentially be the same as for alternating words in type $\tilde{A}_{n-1}$. We write $t=s_0$ and $u=s_n$. Let $[\![i_k,  i_\ell ]\!] $ be a maximal interval such that $|{\bf w}_{s_{i_k}}|=\dots =|{\bf w}_{s_{i_\ell}}| \geq 2$ and $\forall j \in \{0,1, \ldots , n \},$ $ |{\bf w}_{s_j}| \leq |{\bf w}_{s_{i_\ell}}|$. 
Assume $i_\ell \leq n-1$, by maximality of $|w_{s_{i_\ell}}|$ and the fact that $\bf w$ is alternating,  we have $|w_{s_{i_\ell}}|=|w_{s_{i_\ell+1}}|+1$ and there are two possibilities in  $H^c$: $|H^c_{s_{i_\ell}}|= |H^c_{s_{i_\ell-1}}|$ or $|H ^c_{s_{i_\ell}}|= |H^c_{s_{i_{\ell}-1}}| +1$. We obtain either a cylindric convex chain $v \prec_c x \prec_c y \prec_c z$ of length 4 where $v$, $y$ both have label $s_{i_\ell}$ and $x$, $z$ both have label $s_{i_{\ell-1}}$, or a chain covering relation between two indices  $p$ and $q$ with label $s_{i_\ell}$, which  allows us to conclude that $w$ is not  CFC by Theorem~\ref{propcfc} (see figure below, in  each case, we have circled the points $v$, $x$, $y$, $z$ and $q$, $r$ respectively).

\begin{figure}[!h]\includegraphics[scale=1]{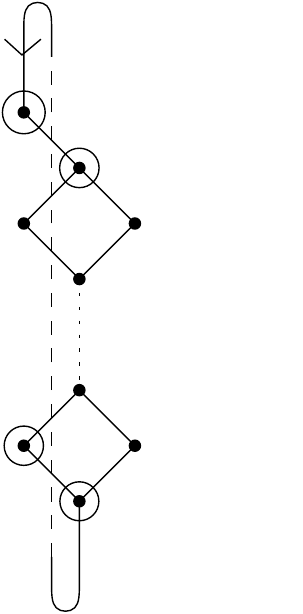} \hspace*{2cm}\includegraphics[scale=1]{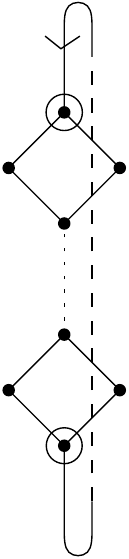}
\caption{The two possible heaps.}
\end{figure}
 So, $i_\ell$ has to be equal to $n$ if $w$ is CFC. The same argument for $i_k$ will lead to $i_k=0$, and so $|{\bf w}_{s_0}|=|{\bf w}_{s_1}|= \dots=|{\bf w}_{s_n-1}| \geq 2$. These elements satisfy (b).
 \end{itemize}
 Next, we will show that among the four remaining families from \cite{BJN2}, the CFC elements must satisfy (c). There are two cases:
 
\begin{itemize}

 \item $w$ belongs to the family called zigzag, i.e $\bf w$ is a subword of the infinite periodic word $(ts_1s_2 \ldots s_{n-1} u s_{n-1} \ldots s_2 s_1)^ {\infty}$ with at least one generator which occurs more than three times (actually, this condition will not be used in this proof). Denote by $s_i$ (\emph{resp.} $s_j$) the first (\emph{resp.} last) letter of $\bf w$. We assume that the second letter is $s_{i-1}$ (the case where the second letter is $s_{i+1}$ is symmetric and can be treated similarly).
 \\ If $s_j \notin \{s_{i-1}, s_{i+1}\}$, 
 $H^c$ contains a cylindric convex chain of length 3 involving points in $H^c_ {\{s_{i-1}, s_{i}\}}$ or a chain covering relation between two points with label $s_i$. 
\\If $s_j$ = $s_{i-1}$, then the last but one letter in $\bf w$ is either $s_i$ or $s_{i-2}$, so $H^c$ takes one of the two forms in Figure~\ref{patternctilde4}. 
 \begin{figure}[!h] \includegraphics[scale=1]{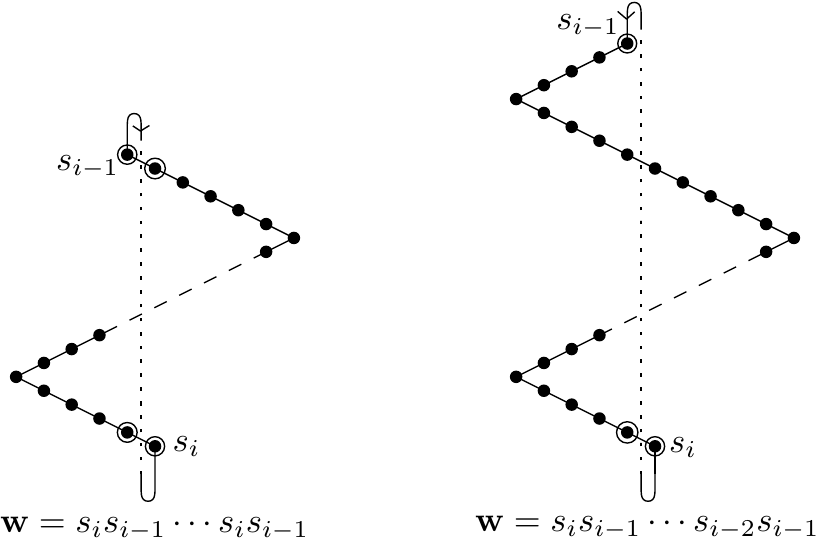}\caption{\label{patternctilde4}The two possible heaps for $s_j=s_{i-1}$.} \end{figure}
\\In the first case, $H^c$ contains a cylindric convex chain involving points in $H^c_ {\{s_{i-1}, s_{i}\}}$ of length 4, given by the two first and the two last letters of $\bf w$ (the corresponding points are circled in Figure~\ref{patternctilde4}, left). By Theorem~\ref{propcfc}, $\bf w$ is not CFC. In the second case, $H^c$ contains a cylindric convex chain involving points in $H^c_ {\{s_{i-1}, s_{i}\}}$ of length 3, given by the two first and the last letters of $\bf w$ (they are circled points in Figure~\ref{patternctilde4}, right). By Theorem~\ref{propcfc}, $\bf w$ is not CFC, unless $s_i=u$, and therefore $\bf w$ satisfies (c).
\\If $s_j$ = $s_{i+1}$,  then the last but one letter in $\bf w$ is either $s_i$ or $s_{i+2}$, so $H^c$ takes one of the two forms in Figure~\ref{patternctilde5}.
\begin{figure}[!h] \includegraphics{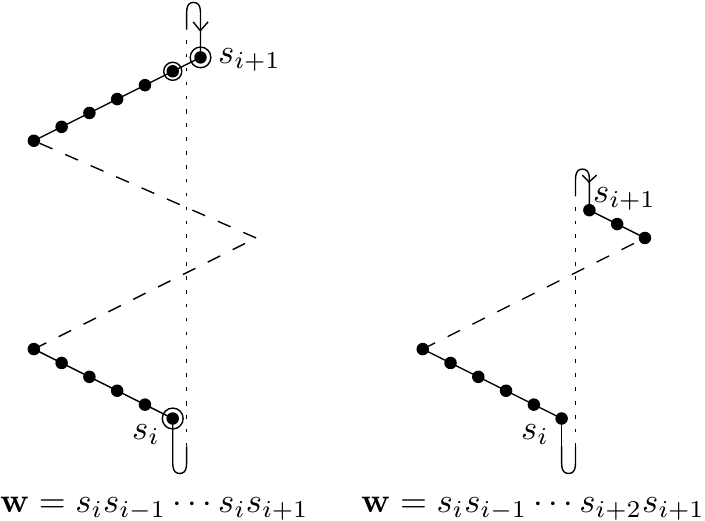}\caption{\label{patternctilde5}The two possible heaps for $s_j=s_{i+1}$.} \end{figure}
  \\In the first case, $H^c$ contains a  cylindric convex chain involving points in $H^c_ {\{s_{i}, s_{i+1}\}}$, of length 3, given by the first and the two last letters of $\bf w$ (the corresponding points are circled in Figure~\ref{patternctilde5}, left). By Theorem ~\ref{propcfc}, $\bf w$ is not CFC, unless $s_i=s_0=t$, which is not possible because the second letter of $\bf w$ is $s_{i-1}$, or unless $s_{i+1}=s_n=u$, and therefore $\bf w$ satisfies (c). In the second case, $\bf w$ satisfies (c).

\item $w$ belongs to one of the three remaining families, called right peak, left peak and left-right peak. In this case, it is easy to find either a chain covering relation $p \prec_c q$ such that $p$ and $q$ have the same label or a cylindric convex chain in $H^c$, except if $w$ satisfies (c) (and has length $2n$).
 
\end{itemize}

Conversely, any $w$ satisfying one of the three conditions (a), (b), or (c) is a CFC element. Indeed, $H^c$ contains neither a relation $i\prec_c j$ such that i and j have the same label nor a cylindric convex chain of length 3, except involving points in $H^c_{\{t, s_1\}}$  or $H^c_{\{s_{n-1}, u\}}$ which do not contain any cylindric convex chain of length 4: for $w$ satisfying (a) or (c), this is trivial, and for $w$ satisfying (b), the same proof as in type $\tilde{A}$ holds (see Theorem~\ref{thmantilde}). \end{proof}

\begin{corollary}
We have the generating function 
\begin{equation}\label{equationctilde}\tilde{C}_n^{CFC}(q) = A_{n+1}^{CFC}(q) + \frac{2^n}{1-q^{n+1}}~q^{2(n+1)}+ \frac{2n}{1-q^{2n}}~q^{2n}.\end{equation}
The coefficients of $\tilde{C}_n^{CFC}(q)$ are ultimately periodic of exact period n(n+1) if n is odd, and 2n(n+1) if n is even. Moreover, periodicity starts at length n.
\end{corollary}
\begin{proof} Notice that the sets of elements satisfying condition (a), (b) or (c) of Theorem~\ref{thmctilde} are disjoint. The first term  in (\ref{equationctilde}) corresponds to the set of CFC elements satisfying (a). This set is clearly in bijection with $A_{n+1}^{CFC}$.
\\Now, we show that the second term  in (\ref{equationctilde}) corresponds to CFC elements satisfying (b). If we fix an integer $h \geq 2$, there are $2^n$ such elements satisfying $h= |{\bf w}_{t}|=|{\bf w}_{s_1}|= \cdots=|{\bf w}_{s_{n-1}}|=|{\bf w}_u|$. Indeed, for each generator $t,~s_1,\ldots,~s_{n-1}$, there are two possibilities: $s_i$ appears before or after $s_{i+1}$ (and $t$ appears before or after $s_1$). The generating function of elements satisfying (b) is then equal to $\displaystyle \sum_{h \geq 2} 2^n q^{h(n+1)}$.
\\ With the condition on the number of occurrences of each generator, we can see that the  length of elements satisfying (c)  must be a multiple of $2n$, and there are $2n$ such elements for each possible length (the first two letters in $\bf w$ can be $s_i s_{i+1} $ or $s_i s_{i-1} $ for all $i \in [\![1,n-1 ]\!],$ or $ts_1$ or $us_{n-1})$. This gives the term $\frac{2n}{1-q^{2n}}~q^{2n}$.
\\$A_{n+1}^{CFC}(q)$ is a polynomial in q, and the two other terms in (\ref{equationctilde}) have  periodic coefficients, so $\tilde{C}_n^{CFC}(q)$ is ultimately periodic. The period is the least common multiple of $n+1$ and $2n$, which is $n(n+1)$ if $n$ is odd, and $2n(n+1)$ if $n$ is even. As the leading coefficient of $A_{n+1}^{CFC}(q)$ is $2^n q^{n+1}$ and $[q^n] A_{n+1}^{CFC}(q) \neq 0$, the beginning of the periodicity is $n$. \end{proof}

We can deduce the characterization in type $B$, whose corresponding Coxeter diagram is recalled below.
 \begin{figure}[h!]
\includegraphics[scale=1]{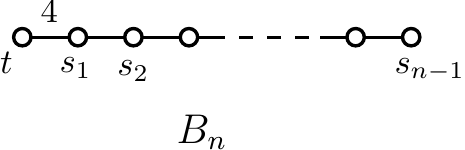}\caption{Coxeter diagram of type $B_n$.}
\end{figure}
\begin{corollary} The CFC elements in type $B_{n }$ are those having reduced expressions in which each generator occurs at most once. Moreover, $$B_n^{CFC}(q)=A_n^{CFC}(q).$$

\end{corollary}

\begin{proof}Let $w$ be a CFC element in type $B_n$, and $\bf w$ one of its reduced expression. Consequently, $w$ is a CFC element in type $\tilde{C_n}$ in which the generator $u$ does not occur. But, according to Theorem \ref{thmantilde}, the only such CFC elements are those satisfying (a) (and having no u): if $w$ satisfies (b) (\emph{resp.} (c)), all generators appear the same number of times (\emph{resp.} u must appear in $\bf w$). Conversely, if all generators occur at most once in $\bf w$, we already saw in Remark~\ref{generatoronce} that $w$ is a CFC element. \end{proof}

\subsection{Types $\tilde{B}_ {n+1}$ and $D_{n+1}$.}
We will obtain once again a characterization of CFC elements in type $\tilde{B}_{n+1}$, deduce a characterization and the enumeration in type $D_{n+1}$, and finally deduce the enumeration in type  $\tilde{B}_{n+1}$.

In what follows, by ``$\bf w$ is an alternating word", we mean that if we replace $t_1$ and $t_2$ by $s_0$, and $u$ by $s_n$, $\bf w$ is an alternating word in the sense of Definition~\ref{alternating}, and $t_1$ and $t_2$ alternate in $\bf w$. 
\newpage
\begin{figure}[!h]
\includegraphics{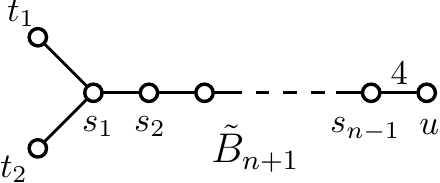} \hspace*{1cm}
\includegraphics{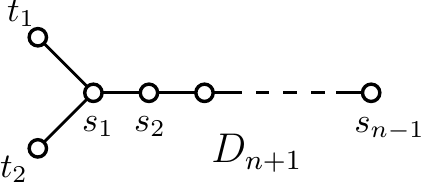}
\caption{Coxeter diagrams of types $\tilde{B}_{n+1}$ and $D_{n+1}$.}
\end{figure}

\begin{Theorem}\label{thmbtilde}
An element $w\in \tilde B_{n+1}$ is CFC if and only if one (equivalently any) of its reduced expressions $\bf w$ verifies one of these conditions:
\begin{itemize}
\item[(a)] each generator occurs at most once in $\bf w$, 
\item[(b)] $\bf w$ is an alternating word,  $|{\bf w}_{s_1}|= \cdots=|{\bf w}_{s_n-1}|=|{\bf w}_u| \geq 2$, and $|{\bf w}_{t_1}|=|{\bf w}_{t_2}|=|{\bf w}_{s_1}|/2$ (in particular, $|{\bf w}_{s_1}|$ is even), 
\item[(c)] $\bf w$ is a subword of $(t_1t_2s_1s_2 \ldots s_{n-1} u s_{n-1} \ldots s_2 s_1)^ {\infty}$, which is an infinite periodic word, where $t_1$ and $t_2$ are allowed to commute, such that $|{\bf w}_{s_1}|= \cdots=|{\bf w}_{s_{n-1}}| \geq 2$, and $|{\bf w}_{t_2}|=|{\bf w}_{t_1}|=|{\bf w}_{u}|=|{\bf w}_{s_1}|/2$ (i.e $\bf w$ takes one of the five forms: $s_is_{i+1} \ldots s_{i-2}s_{i-1}$ or $s_is_{i-1} \ldots s_{i+2}s_{i+1}$ or $t_1s_1 \ldots s_1 t_2$ or $t_2s_1 \ldots t_1$ or $t_1t_2s_1\ldots s_1$). 
\end{itemize}
\end{Theorem}

\begin{proof}
The steps of this proof are the same as in Theorem~\ref{thmctilde}. Let $w$ be a CFC element in $\tilde B_{n+1}$ and let $\bf w$ be one of its reduced expressions. We denote by $H$ the heap of $\bf w$ and $H^c$ its cylindric transformation.
In \cite[Theorem 3.10]{BJN2}, FC elements are classified in five families. The first corresponds to alternating elements. As before, we distinguish two cases:

\begin{itemize} \item each generator occurs at most once in $\bf w$. These elements satisfy (a).
\item $\bf w$ comes from an alternating word of type $\tilde C_n$, where we applied the replacements $(t,t, \ldots, t) \rightarrow (t_1,t_2,t_1, \ldots)$ or $(t_2,t_1, \ldots)$. The same proof as for type $\tilde C$ gives $|{\bf w}_t|=|{\bf w}_{s_1}|= \cdots=|{\bf w}_{s_{n-1}}|=|{\bf w}_u|$. If $|{\bf w}_t|$ is odd, after replacement we have the subword of ${\bf w}$ where we keep only the generators $t_1,~ t_2$ and $s$ is $ t_1s_1t_2s_1 \ldots s_1 t_1$ (or $t_2s_1t_1s_1 \ldots s_1 t_2$), which prevents $w$ from being CFC. So $|{\bf w}_t|$ is even, and as $\bf w$ is alternating, $|{\bf w}_{t_2}|=|{\bf w}_{t_1}|=|{\bf w}_{s_1}|/2$. These elements satisfy (b).
\end{itemize}

 Next, we will show that among the four remaining families from \cite{BJN2}, the CFC elements must satisfy (c). There are two possibilities:
\begin{itemize}

\item $\bf w$ is a subword of $(t_1t_2s_1s_2 \ldots s_{n-1} u s_{n-1} \ldots s_2 s_1)^ {\infty}$, which is an infinite periodic word, where $t_1$ and $t_2$ are allowed to commute, such that a generator occurs more than twice. The same case distinction as in type $\tilde C$ leads us to the required condition (c).
\item $H_{\bf w}$ is a heap among special cases analoguous to the ones in type $\tilde C_n$ which are trivially non CFC, except for those which satisfy (c) (and are of length $2(n+2)$).\end{itemize}
Conversely, all elements $w$ satisfying one of the three conditions (a), (b) or (c) are CFC elements. Indeed, $H^c$ contains neither relations $i\prec j$  such that $i$ and $j$ have the same label nor cylindric convex chains of length 3, except involving points in $H^c_{\{s_{n-1},u\}}$, and these points can not form a cylindric convex chain of length 4. \end{proof}

\begin{proposition} The CFC elements in type $D_{n+1 }$ are those having reduced expressions in which each generator occurs at most once. Moreover, we have $D_{1}^{CFC}(q)=1+q$, $D_{2}^{CFC}(q)=1+2q+q^2$, $D_{3}^{CFC}(q)=1+3q+5q^2+4q^3$ and for $n \geq 3$:
 \begin{equation}\label{equationd}D_{n+1}^{CFC}(q) = (2q+1)D_{n}^{CFC}(q)-qD_{n-1}^{CFC}(q).\end{equation}
In other words, we have
 $$D^{CFC}(x):=\sum_{n=0}^{\infty} D_{n+1}^{CFC}(q)x^n= \frac{(1+q)-xq(1+q)+x^2q^2(1+2q)}{1-(2q+1)x+qx^2}.$$
\end{proposition}

\begin{proof} Let $ w$ be a CFC element of type $D_{n+1}$ and let $\bf w$ be one of its reduced expressions. Consequently, $w$ is a CFC element of type $\tilde{B}_{n+1}$ in which the generator $u$ does not occur. But, according to Theorem~\ref{thmbtilde}, the only such CFC elements are those satisfying condition (a) (and having no u): if $\bf w$ satisfies (b) or (c), u must appear in $\bf w$. Conversely, if all generators occur at most once in $\bf w$, we already saw that $w$ is CFC. 
\\So, any CFC element of type $D_{n+1}$ can be otained from a CFC element $w'$ in $A_{n-1}$ by adding to $w'$ nothing, or one occurrence of $t_1$ (with two choices: before or after $s_1$), or one occurrence of $t_2$ (with two choices), or one occurrence of $t_1$ and $t_2$ (with four choices).  If $s_1$ does not occur in $w'$, we have no choice for adding $t_1$ or/and $t_2$ as $t_1$ and $t_2$ commute with all other generators. This leads to the following recurrence relation:
$$D_{n+1}^{CFC}(q)=(1+4q+4q^2) A_{n-1}^{CFC}(q) -(2q+3q^2)A_{n-2}^{CFC}(q).$$

With this relation and (\ref{equationa2}) in Theorem \ref{thman}, it is easy to derive (\ref{equationd}). The generating function $D^{CFC}(x)$ is computed through classical methods. \end{proof}

\begin{corollary}
We have the generating function 
\begin{equation}\label{equationbtilde}\tilde{B}_{n+1}^{CFC}(q) = D_{n+2}^{CFC}(q) + \frac{2^{n+1}}{1-q^{2(n+1)}}~q^{2(n+1)}+ \frac{2(n+1)}{1-q^{2n+1}}~q^{2n+1}.\end{equation}
Furthermore, the coefficients of $\tilde{B}_{n+1}^{CFC}(q)$ are ultimately periodic of exact period \\2(n+1)(2n+1) and the periodicity starts at length $n+2$.
\end{corollary}

\begin{proof}
Notice that the sets of elements satisfying condition (a), (b) or (c) of Theorem~\ref{thmbtilde} are disjoint. The first term  in (\ref{equationbtilde}) corresponds to the set of CFC elements satisfying (a). This set is clearly in bijection with $D_{n+2}^{CFC}$.\\ The second term  in (\ref{equationbtilde}) corresponds to CFC elements satisfying (b), which have length a multiple of $2(n+1)$, and there are $2^{n+1}$ such elements for each possible length.\\ With the condition on the number of occurrences of each generator, we can see that the  length of elements satisfying (c)  must be a multiple of $2n+1$, and there are $2n+2$ such elements for each possible length (the first two letters can be $s_is_{i+1}$, $s_is_{i-1}$, $us_{n-1}$, $t_1t_2$, $t_1s_1$ or $t_2s_1$). This gives the third term of \eqref{equationbtilde}.
\\Moreover, $D_{n+2}^{CFC}(q)$ is a polynomial in q, and the two other terms of (\ref{equationbtilde}) are periodic, so $\tilde{B}_{n+1}^{CFC}(q)$ is ultimately periodic. The period is the least common multiple of $2(n+1)$ and $2n+1$, which is $2(n+1)(2n+1)$. The beginning of the periodicity is $n+2$, as the leading coefficient of $D_{n+2}^{CFC}(q)$ is $2^{n+1} q^{n+2}$ and $[q^{n+1}]D_{n+2}^{CFC}(q) \neq 0$. \end{proof}

\subsection{Type $\tilde{D}_{n+2}$.}

The situation is very similar in type $\tilde D_{n+2}$ for the characterization, but the generating function takes a slightly different form, due to the specificity of the Coxeter diagram. 
\begin{figure}[!h]
\begin{center}\includegraphics[scale=1]{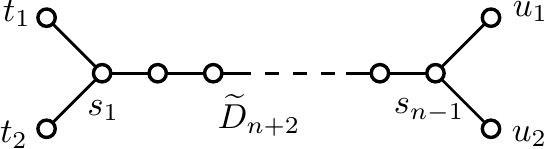}
\end{center}
\caption{Coxeter diagram of type $\tilde{D}_{n+2}$.}
\label{dynkindtilde}
\end{figure}

In what follows, by ``$\bf w$ is an alternating word", we mean that if we replace $t_1$ and $t_2$ by $s_0$, and $u_1$ and $u_2$ by $s_n$, the image of ${\bf w}$ is an alternating word in the sense of Definition~\ref{alternating}, $t_1$ and $t_2$ alternate in $w$, and $u_1$ and $u_2$ alternate in $w$. 
\begin{Theorem}\label{thmdtilde}
An element $w\in \tilde D_{n+2}$ is CFC  if and only if one (equivalently any) of its reduced expressions  \textbf{w}  verifies one of these conditions:
\begin{itemize} 
\item[(a)] each generator occurs at most once in $\bf{w}$, 
\item[(b)] $\bf w$ is an alternating word,  $|{\bf w}_{s_1}|= \cdots=|{\bf w}_{s_n-1}| \geq 2$, and $|{\bf w}_{t_1}|=|{\bf w}_{t_2}|=|{\bf w}_{u_1}|=|{\bf w}_{u_2}|=|{\bf w}_{s_1}|/2$ (in particular, $|{\bf w}_{s_1}|$ is even), 
\item[(c)] $\bf w$ is a subword of $(t_1t_2s_1s_2 \ldots s_{n-1} u_1 u_2 s_{n-1} \ldots s_2 s_1)^ {\infty}$, which is an infinite periodic word, where $t_1$ and $t_2$, $u_1$ and $u_2$ are allowed to commute, such that $|{\bf w}_{s_1}|= \cdots=|{\bf w}_{s_{n-1}}| \geq 2$, and $|{\bf w}_{t_2}|=|{\bf w}_{t_1}|=|{\bf w}_{u_1}|=|{\bf w}_{u_2}|=|{\bf w}_{s_1}|/2$.
\end{itemize}
\end{Theorem}

\begin{proof} The same proof as for Theorem~\ref{thmbtilde} holds, one only needs to add the replacements $(u,u,\ldots u)$ $\rightarrow(u_1, u_2, \ldots)$ or $(u_2, u_1, \ldots)$. \end{proof}

Again, we can compute the corresponding generating function.

\begin{proposition}
We have the generating function 
\begin{equation}\label{equationdtilde}\tilde{D}_{n+2}^{CFC}(q) = Q_{n+2}(q) + \frac{2^{n+2}+2(n+2)}{1-q^{2(n+1)}}~q^{2(n+1)},\end{equation}
where $Q_{n+2}(q)$ is a polynomial in q of degree $n+3$ such that $Q_4(q) = 1+5q+14q^2+28q^3+33q^4+16q^5$, $Q_5(q) =1+6q+20q^2+46q^3+73q^4+72q^5+32q^6$, and for $n \geq 4$:
\begin{equation*}Q_{n+2}(q)=(2q+1)Q_{n+1}(q) -qQ_n(q). \end{equation*}
Moreover, the coefficients of $\tilde{D}_{n+2}^{CFC}(q)$ are ultimately periodic of exact period 2(n+1), and the periodicity starts at length $n+2$.
\end{proposition}
Note that the generating function $Q(x)$ of the polynomials $Q_{n}(q)$ is computable through classical techniques. However it does not have a nice expression.
\begin{proof} Notice that the sets of elements satisfying condition (a), (b) or (c) of Theorem~\ref{thmdtilde} are disjoint. The first term  in (\ref{equationdtilde}) corresponds to the set of CFC elements satisfying (a). Any such element can be otained from a CFC element $w$ of type $D_{n+1}$ by adding nothing, or one occurrence of $u_1$ (with two choices: before or after $s_{n-1}$), or one occurrence of $u_2$ (with two choices), or one occurrence of $u_1$ and $u_2$ (with four choices). If $s_{n-1}$ does not occur in $w$, we have no choice for adding $u_1$ or/and $u_2$, as $u_1$ and $u_2$ commute with all other generators. This leads to the following recurrence relation:
$$Q_{n+2}(q)=(1+4q+4q^2) D_{n+1}^{CFC}(q) -(2q+3q^2)D_{n}^{CFC}(q).$$
Using this and \eqref{equationd}, we find the expected recurrence relation for the polynomials $Q_n(q)$.
 Elements satisfying (b) have length a multiple of $2(n+1)$ and there are $2^{n+2}$ of them. Elements satisfying (c) have also length a multiple of $2(n+1)$, and there are $2(n+2)$ of them (by inspecting the first two letters). The property of ultimate periodicity is clear. \end{proof}

\subsection{Exceptional types}
Exceptional finite types are $E_6,~E_7,~E_8,~F_4,~H_2,~H_3,$ $~H_4,~G_2,~I_2(m)$. Enumerating CFC elements according to the length in these two last ones is trivial, while other groups are special cases of the families $E_n$ ($n\geq 6$), $F_n$ ($n\geq 4$), $H_n$ ($n \geq 3$). It is shown in~\cite{BBEEGM} that these families contain a finite number of CFC elements. It is possible to apply our method to characterize them in terms of cylindric heaps and obtain recurrence relations for their generating functions according to the length. However, we leave this to the interested reader.
For exceptional affine types having a finite number of CFC elements ($\tilde{E}_8$ and $\tilde{F}_4$), generating functions of CFC elements are polynomials recursively computable as in type $A$, because $\tilde{E}_8=E_9$ and $\tilde{F}_4=F_5$. It remains to study three exceptional affine types having an infinite number of CFC elements. Their Coxeter diagrams are represented below. For these types, we only give the result without detailing the proof. For the interested readers, we just sketch it: as for other affine types, we look at the classification of FC elements in \cite[Lemmas 5.2--5.4]{BJN2}, and see that there exists CFC elements only for some explicit length, which must be either bounded or an integer multiple of a constant depending on the considered type.

\begin{figure}[!h]
\includegraphics{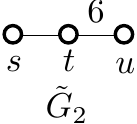}\hspace*{1cm}
\includegraphics{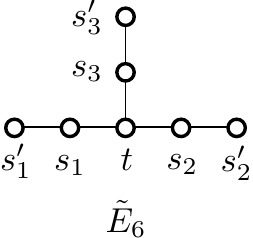} \hspace*{1cm}
\includegraphics[scale=1]{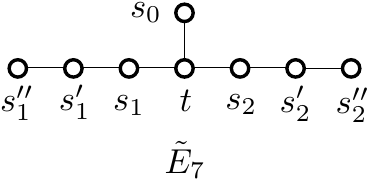}
\caption{Coxeter diagrams of types $\tilde{G}_{2}$, $\tilde{E}_{6}$ and $\tilde{E}_{7}$.}
\end{figure}

\begin{Theorem} Let $W= \tilde{G}_2$ (\emph{resp.} $\tilde{E}_6$, \emph{resp.} $\tilde{E}_7$) be an exceptional affine Coxeter group. We define in $W$ the element $w_2$ as $utsut$ (\emph{resp.} $s_1ts_1's_1s_2ts_2's_2s_3ts_3's_3$, \emph{resp.}  $s_1ts_1's_0s_1s_1''ts_1's_1s_2ts_2's_0s_2s_2''ts_2's_2$). An element $w\in W$ is CFC if and only if one (equivalently any) of its reduced expressions $\bf w$ verifies one of these conditions:
\begin{itemize}
\item $\bf w$ belongs to a finite set depending on $W$,
\item ${\bf w}= {\bf w}_1 {\bf w}_2 ^n {\bf w}_3$, where n is a non-negative integer, ${\bf w_2} \in R(w_2)$ and $\bf w_3w_1=w_2$.
\end{itemize}
Consequently, the following generating functions hold:
\begin{equation*}\tilde{G}_2^{CFC}(q)= R_1(q) +\frac{6q^5}{1-q^5},
\end{equation*}
\begin{equation*}\tilde{E}_6^{CFC}(q)= R_2(q) +\frac{23q^{12}}{1-q^{12}},
\end{equation*}
\begin{equation*}\tilde{E}_7^{CFC}(q)= R_3(q) +\frac{45q^{18}}{1-q^{18}},
\end{equation*}
where $R_1(q),~R_2(q)$ and  $R_3(q)$ are polynomials.
\end{Theorem}

\subsection{Logarithmic CFC elements}\label{logarithmic}

Here we study some particular elements of a general Coxeter group $W$. Recall that for $w \in W^{CFC}$, the support $supp(w)$ of $w$ is the set of generators that appear in a (equivalently, any) reduced expression of $w$. It is well known that if ${ w} \in W$ and $k$ is a positive integer, then $\ell({ w}^k) \leq k \ell({ w})$. If equality holds for all $k \in \mathbb{N}^*$ ($i.e$ ${ w}^k$ is reduced for all k), we say that $ w$ is logarithmic (see \cite{BBEEGM} for more information about logarithmic elements). For example, it is trivial that there is no logarithmic element in finite Coxeter groups, because there is a finite number of reduced expressions. By using our characterizations for affine types in Theorems~\ref{thmantilde}, \ref{thmbtilde}, \ref{thmctilde}, \ref{thmdtilde},  we derive the following consequence. A generalization of this result was proved for all Coxeter groups in \cite[Corollary~E]{MAR}, by using geometric group theoretic methods. Nevertheless, as our approach is different and more combinatorial, we find interesting to give our proof, although it only works for affine types.

\begin{Theorem} For $W$=$ \tilde A, \tilde B , \tilde C,$ or $\tilde D$,  if $w$ is a CFC element, $w$ is logarithmic if and only if a (equivalently, any) reduced expression ${\bf w}$ of $w$ has full support (i.e all generators occur in ${\bf w}$). In particular, there is a finite number of CFC elements which are not logarithmic.
\end{Theorem}
\begin{proof} Let $w$ be in $W^{CFC}$ and let $\bf w$ be one of its reduced expressions. Assume that ${\bf w}$ has not full support. In this case, $w$ belongs to a proper parabolic subgroup of $W$, which is a finite Coxeter group by a classical property of affine Coxeter groups (see  \cite{HUM}). So $w$ is not logarithmic, and there is a finite number of such elements, which correspond to elements satisfying (a) in Theorems~\ref{thmantilde}, \ref{thmbtilde}, \ref{thmctilde}, \ref{thmdtilde} and such that at least one generator does not occur in $\bf w$.
\\Conversely, assume that $w$ is a CFC element with full support. According to Theorems~\ref{thmantilde}, \ref{thmbtilde}, \ref{thmctilde}, \ref{thmdtilde},  ${\bf w}$ must satisfy (b) or (c) or has to be a Coxeter element ($i.e$ each generator occurs exactly once in {\bf w}). If $w$ is a Coxeter element, $w$ is logarithmic (see \cite{SPE}). If $\bf w$ satisfies (b) (\emph{resp.} (c)), it is easy to verify that ${\bf w}^k$ also satisfies (b) (\emph{resp.} (c)) and  is therefore reduced. It follows that $w$ is logarithmic.  \end{proof}
We can also notice that the powers of Coxeter elements are always CFC in affine types $\tilde{A}_{n-1}$ and $\tilde{C}_n$ because they satisfy the alternating word condition, but are never CFC elements in types $\tilde{B}_{n+1}$ and $\tilde{D}_{n+2}$ (because they satisfy neither the alternating word condition nor condition (c)).

\section{CFC involutions} A natural question that arises in the study of FC elements is to compute the number of FC involutions in finite and affine Coxeter groups (as this number is, for some groups, the sum of the dimensions of irreducible representations of a natural quotient of the Iwahori-Hecke algebra associated to the group, see \cite{FAN}). Similarly, we now focus on CFC elements which are involutions. The main result is that there is always a finite number of such elements in all Coxeter groups, and we are able to characterize them in terms of words. We also use the characterization of CFC elements to enumerate CFC involutions in finite and affine Coxeter groups.

\subsection{Finiteness and characterization of CFC involutions}

\begin{Theorem}\label{finiteness}
Let $W$ be a Coxeter group and let I($W$) be its subset of involutions. The set $W^{CFC}\cap $I($W$)  is finite. Moreover an element $w$ belongs to $W^{CFC}\cap $I($W$) if and only if for any generator s in  $supp(w)$, $|w_s|=1$ and for all t such that $m_{st} \geq 3$, $|w_t| =0$ (i.e two non commuting generators can not occur in $w$).
\end{Theorem}

\begin{proof}  Let $w$ be a CFC involution, let s be a generator in $supp(w)$ and $\bf w$ be a reduced expression of $w$. Assume that $|w_s| \geq 2$. Consider a cyclic shift ${\bf w'}=s{\bf w_1}$  of $\bf w$ which begins with s. As $w$ and s are involutions,  $\bf w'$ is an expression of an involution. Moreover as $w$ is CFC, $\bf w'$ corresponds to a FC element $w'$. According to \cite{STEM1}, a FC element $w$ is an involution if and only if R($w$) is palindromic, i.e R($w$) includes the reverse of all of its members. Applying this to $\bf w'$ allows us to say that $\bf w'$ is commutation equivalent to a word $s{\bf w_2} s$. So a cyclic shift of $\bf w$ is commutation equivalent to $ss{\bf w_2}$, which is in contradiction with the CFC property of $w$. Therefore we have $|w_s| =1$. We consider $\bf w'$ as before. As R($w'$) is palindromic, we can conclude that all generators in $supp(w')$=$supp(w)$ commute with s. \end{proof}

\begin{Remark}\label{remarkinv}
The number of CFC involutions in a Coxeter group $W$ depends only on the edges of the Coxeter diagram, without taking into account  the values $m_{st}$.
\end{Remark}

\subsection{CFC involutions in classical types}

Let us enumerate the CFC involutions in classical types, according to their Coxeter length. If $W$ is a Coxeter group, we define $W^{CFCI}(q):=\sum_{w \in W^{CFC}\cap I(W)}q^{\ell(w)}$.

\begin{Theorem}\label{ainvo} In types $A$, $B$, and $\tilde{C}$ the following relation holds for all $n\geq2$:
\begin{equation}\label{equationinva}A^{CFCI}_{n}(q)=B^{CFCI}_{n}(q)= \tilde{C}^{CFCI}_{n-1}(q)=A^{CFCI}_{n-1}(q)+qA^{CFCI}_{n-2}(q),\end{equation}
and $A^{CFCI}_{0}(q)=1$, $A^{CFCI}_{1}(q)=1+q$.
Moreover, we can compute the generating function 
$$A^{CFCI}(x):=\sum_{n=1}^\infty A^{CFCI}_{n-1}(q)x^{n}=x\frac{1+qx}{1-x-qx^2}.$$

\end{Theorem}
\begin{proof} The equality $A^{CFCI}_{n}(q)=B^{CFCI}_{n}(q)= \tilde{C}^{CFCI}_{n-1}(q)$ comes from Remark~\ref{remarkinv}. Let $w$ be a CFC involution in $A_n$. If $s_{n}$ belongs to $supp(w)$, by Theorem~\ref{finiteness}, $s_{n-1}$ does not belong to $supp$($w$), and $w$ is equal to $s_{n}$ concatenated to a CFC involution of $A_{n-2}$.  If $s_{n}$ does not belong to $supp$($w$), $w$ is a CFC involution of $A_{n-1}$. This yields the expected relation \eqref{equationinva}. The generating function is computed through classical techniques. \end{proof}

In particular, if $q \to 1$, we obtain the number of CFC involutions in type $A_{n-1}$, which is the $(n+1)^{th}$ Fibonacci number.

\begin{Theorem}
In types D and $\tilde{B}$, the following relations hold for all $n\geq 3$:
\begin{equation}\label{equationinvd}D^{CFCI}_{n}(q)= \tilde{B}^{CFCI}_{n-1}(q)=D^{CFCI}_{n-1}(q)+qD^{CFCI}_{n-2}(q),\end{equation}
\begin{equation}\label{equationinvd2}D^{CFCI}_{n+1}(q)=qA^{CFCI}_{n-3}(q) + (1+2q+q^2) A^{CFCI}_{n-2}(q),  \end{equation}
and $D^{CFCI}_{1}(q):=1$, $D^{CFCI}_{2}(q)=1+2q+q^2$, $D^{CFCI}_{3}(q)=1+3q+q^2$.
Moreover, we can compute the generating function $$D^{CFCI}(x):=\sum_{n=0}^\infty D^{CFCI}_{n+1}(q)x^n=\frac{1+(2q+q^2)x}{1-x-qx^2}.$$
\end{Theorem}

\begin{proof} To prove \eqref{equationinvd2}, let $w$ be a CFC involution in $D_{n+1}$. If $s_{1}$ belongs to $supp$($w$), by Theorem~\ref{finiteness}, $s_{2}, t_1, t_2$ do not belong to $supp$($w$), and $w$ is equal to $s_{1}$ concatenated to a CFC involution of $A_{n-3}$.  If $s_{1}$ does not belong to $supp$($w$), $w$ is a CFC involution of $A_{n-1}$ concatenated to $t_1$, $t_2$, $t_1t_2$ or nothing. This shows \eqref{equationinvd2}. Equation \eqref{equationinvd} is simply a consequence of \eqref{equationinvd2},  \eqref{equationinva} and Remark~\ref{remarkinv}. The generating function comes from classical techniques.  \end{proof}

If $q \to 1$, we obtain the number of CFC involutions in type $D_{n+1}$, which leads to a Fibonacci-type sequence with starting numbers 1 and 4.

\begin{Theorem} In type $\tilde{A}$, the following relations hold for $n \geq 2$:
\begin{equation}\label{equationinvatilde}\tilde{A}^{CFCI}_{n}(q)=\tilde{A}^{CFCI}_{n-1}(q)+q\tilde{A}^{CFCI}_{n-2}(q),\end{equation}
where $\tilde{A}^{CFCI}_{0}(q):=1$, $\tilde{A}^{CFCI}_{1}(q)=1+2q$.
Therefore, we can compute the generating function 
$$\tilde{A}^{CFCI}(x):=\sum_{n=1}^\infty \tilde{A}^{CFCI}_{n-1}(q)x^n=x\frac{1+2qx}{1-x-qx^2}.$$
\end{Theorem}
\begin{Remark}We also have
$$q^n \tilde{A}^{CFCI}_{n-1}(1/q^2)=L_n(q),$$
where $L_n(q)$ is the $n^{th}$ Lucas polynomials, defined explicitely (see sequence A114525 in \cite{SLO}) by 
$L_n(q) =2^{-n}[(q-\sqrt{q^2+4})^n+(q+\sqrt{q^2+4})^n]$. This equality can be proved by using generating functions.
\end{Remark}
\begin{proof}Let $w$ be a CFC involution in $\tilde{A}_{n-1}$. If $s_{0}$ belongs to $supp$($w$), by Theorem~\ref{finiteness}, $s_{n-1}$ and $s_1$ do not belong to $supp$($w$), and $w$ is equal to $s_{0}$ concatenated to a CFC involution of $A_{n-3}$.  If $s_{0}$ does not belong to $supp$($w$), $w$ is a CFC involution of $A_{n-1}$. This leads to the relation:$$\tilde{A}^{CFCI}_{n-1}(q)=qA^{CFCI}_{n-3}(q)+A^{CFCI}_{n-1}(q).$$ Using this and  (\ref{equationinva}), it is easy to deduce (\ref{equationinvatilde}). The generating function comes from classical techniques.  \end{proof}

\begin{Theorem}In type $\tilde{D}$, the following relation holds for $n \geq 4$:
\begin{equation}\label{equationinvdtilde}\tilde{D}^{CFCI}_{n+2}(q)=\tilde{D}^{CFCI}_{n+1}(q)+q\tilde{D}^{CFCI}_{n}(q), \end{equation}
with $\tilde{D}^{CFCI}_{4}(q)=1+5q+6q^2+4q^3+q^4$ and $\tilde{D}^{CFCI}_{5}(q)=1+6q+10q^2+6q^3+q^4$.
We can therefore compute the generating function:
$$\tilde{D}^{CFCI}(x):=\sum_{n=2} ^\infty \tilde{D}^{CFCI}_{n+2}(q)x^n=x^2\frac{1+5q+6q^2+4q^3+q^4+x(q+4q^2+2q^3)}{1-x-qx^2}.$$
\end{Theorem}
\begin{proof}  Let $w$ be a CFC involution in $\tilde{D}_{n+2}$. If $s_{1}$ belongs to $supp$($w$), by Theorem~\ref{finiteness},  $s_2$, $t_1$ and $t_2$ do not belong to $supp$($w$), and $w$ is equal to $s_{1}$ concatenated to a CFC involution of $D_{n-1}$.  If $s_{1}$ does not belong to $supp$($w$), $w$ is a CFC involution $D_{n}$ concatenated to $t_1$, $t_2$, $t_1t_2$ or nothing. This leads to the relation:$$\tilde{D}^{CFCI}_{n+2}(q)=qD^{CFCI}_{n-1}(q)+(1+2q+q^2) D^{CFCI}_{n}(q).$$
Use this and (\ref{equationinvd}) to derive (\ref{equationinvdtilde}). 
\end{proof}
If $q \to 1$, we obtain the number of involutions in type $\tilde{D}_{n+2}$, which is a Fibonacci-type sequence with starting numbers 10 and 7.

\section{Other questions}

In Sections~\ref{finite} and \ref{affine}, we obtained two $q$-analogs of the number of CFC elements in finite types. The first, in type $A$, corresponds to permutations avoiding 321 and 3412, taking into account their Coxeter length (see \cite{TEN}). The second, in type $D$, is apparently new. One may wonder if it corresponds to another combinatorial object.
\\We can notice that if $W$ is a finite or affine Coxeter group, the generating function of CFC elements is rational. This is in fact true for all Coxeter groups, as will be proved in the forthcoming work \cite{P}.

\section{Acknowledgements}
We would like to thank M. Macauley for pointing out reference~\cite{MAR} on logarithmic elements, and T. Marquis for valuable remarks on the infinite number of non logarithmic CFC elements. Finally, we also thank P. Nadeau for interesting comments and discussions on an earlier version of this manuscript.

\bibliographystyle{plain}

\begin{thebibliography}{SGA}





\bibitem[1]{BJN2}
R. Biagioli, F. Jouhet, and P. Nadeau, \textit{Fully commutative elements in finite and affine Coxeter groups}, to appear in Monatsh. Math., preprint arXiv:1402.2166 (2014), 37 pages. 

\bibitem[2]{BB}
A. Bj\"orner, and F. Brenti, \textit{Combinatorics of Coxeter groups}, Springer (2005).


\bibitem[3]{BBEEGM}
T. Boothby, J. Burkert, M. Eichwald, D. C. Ernst, R. M. Green, and M. Macauley, \textit{On the cyclically fully commutative elements of Coxeter groups}, J. Algebraic Combin. {\bf 1} (2012), 123--148. 

\bibitem[4]{FAN2}
C. K. Fan, \textit{A Hecke algebra quotient and properties of commutative elements of a Weyl group.}, Phd thesis, M.I.T. (1995).


\bibitem[5]{FAN}
C. K. Fan, \textit{Structure of a Hecke Algebra Quotient}, J. Amer. Math. Soc. {\bf 1}(1997), 139--167.

\bibitem[6]{GRA} 
J. Graham, \textit{Modular Representations of Hecke Algebras and Related Algebras}, PhD thesis, University of Sydney (1995).


\bibitem[7]{HUM}
J. E. Humphreys, \textit{Reflection groups and Coxeter groups}, Volume 29 of
Cambridge Studies in Advanced Mathematics, Cambridge University Press, Cambridge (1990).

\bibitem[8]{MAR}
T. Marquis, \textit{Conjugacy classes and straight elements in Coxeter groups}, to appear in Journal of Algebra, preprint arXiv:1310.1021v2 (2014), 12 pages.

\bibitem[9]{P}
M. P\'etr\'eolle, \textit{Rational generating function for cyclically fully commutative elements}, in preparation (2014).

\bibitem[10]{SLO} 
N. J. A. Sloane, \textit{The on-line encyclopedia of integer sequences}  (2013).

\bibitem[11]{STEM1}
J. R. Stembridge, \textit{On the fully commutative elements of Coxeter groups}, J. Algebraic Combin. {\bf 4} (1996), 353--385. 

\bibitem[12]{STEM2}
J. R. Stembridge, \textit{Some combinatorial aspects of reduced words in finite Coxeter groups}, Trans. Amer. Math. Soc. {\bf 4} (1997), 1285--1332. 

\bibitem[13]{STEM3}
J. R. Stembridge, \textit{The enumeration of fully commutative elements of Coxeter groups}, J. Algebraic Combin. {\bf 7} (1998), 291--320. 

\bibitem[14]{SPE} 
D. E. Speyer, \textit{Powers of Coxeter elements in infinite groups are reduced}
Proc. Amer. Math. Soc. {\bf 4} (2009), 1295--1302. 

\bibitem[15]{TEN} 
B. E. Tenner, \textit{Pattern avoidance and the Bruhat order}, J. Combin. Theory Ser. A {\bf 5} (2007), 888--905.

\end{thebibliography}

\end{document}